\documentclass[
a4paper,
11pt,
twoside,
]{article}
\usepackage[utf8]{inputenc}
\usepackage[T1]{fontenc}
\usepackage[sc]{mathpazo}
\usepackage[english]{babel}
\usepackage{geometry}
\usepackage{amsmath,amsfonts,amssymb,amsthm}
\usepackage[final]{hyperref}
\usepackage{graphicx}
\usepackage{mathtools}
\usepackage[dvipsnames]{xcolor}
\usepackage{pdflscape}
\usepackage{etoolbox}
\usepackage{fancyhdr}
\usepackage{lastpage}
\usepackage{ifdraft}
\usepackage{hyphenat}
\ifdraft{\usepackage{showlabels}}{}
\usepackage{enumitem}
\usepackage{siunitx}
\usepackage[mathlines,pagewise]{lineno}

\newcommand*{\diff}{\mathop{}\!\mathrm{d}}
\newcommand{\R}{\mathbb{R}}

\theoremstyle{plain}
\newtheorem{theorem}{Theorem}[section]
\newtheorem{lemma}[theorem]{Lemma}

\newtheorem{assumption}[theorem]{Asssumption}

\newtheorem{remark}[theorem]{Remark}

\numberwithin{equation}{section}
\numberwithin{table}{section}
\numberwithin{figure}{section}

\fancypagestyle{mypagestyle}{%
  \fancyhf{}%
  \fancyhead[LO,RE]{\scriptsize expRK methods for delay equations}%
  \fancyhead[RO,LE]{\scriptsize A. And\`o and R. Vermiglio}%
  \fancyfoot[LE,RO]{\scriptsize \thepage\,/\,\pageref*{LastPage}}%
}
\title{Exponential Runge-Kutta methods for delay equations in the sun-star abstract framework}
\author{
Alessia And\`o$^{1,2}$ and Rossana Vermiglio$^{1,3}$\\[.5em]
\small $^{1}$CDLab -- Computational Dynamics Laboratory\\[-.2em]
\small Department of Mathematics, Computer Science and Physics -- University of Udine\\[-.2em]
\small via delle scienze 206, 33100 Udine, Italy\\[.5em]
\small $^{2}$\texttt{alessia.ando@uniud.it}\\[.5em]
\small $^{3}$\texttt{rossana.vermiglio@uniud.it}
}
\date{\today}
\begin{document}
\clearpage
\maketitle
\thispagestyle{empty}
%-----------------------------------------------------------------------------
\begin{abstract}
Exponential Runge-Kutta methods for semilinear ordinary differential equations can be extended to abstract differential equations, defined on Banach spaces.
Thanks to the sun-star theory, both delay differential equations and renewal equations can be recast as abstract differential equations, which motivates the present work. The result is a general approach that allows us to define the methods explicitly and analyze their convergence properties in a unifying way. 
\end{abstract}

\smallskip
\noindent {\bf{Keywords:}} exponential methods, sun-star theory, abstract integral equations, delay differential equations, renewal equations.

\smallskip
\noindent{\bf{2020 Mathematics Subject Classification:}} Primary: 65L03, 65L04, 65L06, 65R20; Secondary: 34K05, 45D05, 47D03
%-----------------------------------------------------------------------------
\section{Introduction}\label{S_intro}
Numerical simulations play an important role in improving our understanding and insight into the behaviour of complex dynamical systems.\\
Although exponential integrators were already considered in the sixties, only in the recent decades they have proved to be effective numerical methods for time integration of stiff differential equations, and extensive research has been devoted to their construction, analysis, implementation and application. Good reviews can be found in \cite{HO10,MW05}. In various scientific and engineering applications, both the abstract formulation of time dependent partial differential equations in some Banach space and their spatial discretization give rise to semi-linear evolution problems of the form
\begin{equation}\label{IVPADE}
\left\{\setlength\arraycolsep{0.1em}\begin{array}{rcl}
u'(t)&=&\mathcal{A}_0u(t)+\mathcal{F}(t,u(t)),\quad t\in[0,T],
 \\[1mm]
u(0)&=&u_0,
\end{array}
\right.
\end{equation} 
where $u_0$ belongs to a Banach space $X$, $\mathcal{A}_0:\text{ dom}(\mathcal{A}_0)\subset X\to X$ is an unbounded operator or a stiff matrix, while $\mathcal{F}$ satisfies a local Lipschitz condition w.r.t. $u$ with a moderate Lipschitz constant in a strip along the exact solution. 

The basic idea of exponential integrators for \eqref{IVPADE} is to treat the nonlinearity explicitly and the linear part exactly. In particular, analogously to the construction of Runge-Kutta (RK) methods for ordinary differential equations (ODEs),  explicit exponential Runge-Kutta (ExpRK) methods have been derived starting from the variation of constants equation, employing suitable exponential quadrature rules. Extensive studies have been carried out to analyze convergence, and to construct  high-order methods \cite{HO10,L21}. These schemes allow for relative large step size and they have been applied to address challenges posed by parabolic problems \cite{DJLQ21,HO05a,HO05b,LO14} as well as other type of problems. As the list of references is quite long, we limit it to a selection of instances \cite{DP11,DZ04,D09,LL24,OSV23}.

%   or by using efficient algorithm for the product of matrix exponential of $\mathcal{A}_0$ with a vector.

In this paper we focus on delay equations, which, according to a well-known definition, are {\it rules for extending a function of time towards the future on the basis of the (assumed to be) known past} \cite{twin}.  When the rule of extension specifies the derivative of the unknown function at a certain time in terms of its past values, the delay equation is a delay differential equation (DDE), while in the case of renewal equations (REs) it is the value of the unknown function at a given time to be prescribed in term of the past values. Delay equations, including coupled DDEs/REs, are widely used in modeling physical and biological phenomena characterized by time delays or memory effects in their dynamics, see for instance \cite{diekmann95} and the refences therein. The introduction of the delay makes the model more accurate in predicting the real-life phenomena but generates infinite-dimensional dynamical systems. Therefore, numerical methods are essential to investigate their dynamics, with numerical simulations playing a crucial role. By means of the sun-star theory, delay equations can be reformulated as semilinear abstract equations of the form \eqref{IVPADE}, and an abstract variation-of-constants equation can be derived \cite{twin,diekmann95}. Thus, the sun-star theory provides the fundamental ingredients to apply explicit ExpRK-methods for time integration of delay equations in this abstract setting. This motivates us to consider these schemes for time integration of delay equations and explore their convergence in an unifying and original way. We remark that explicit ExpRK methods have been applied directly to a certain class of semilinear DDEs in \cite{FZ21,ZZO16} and the sun-star abstract setting has not been considered.

The outline of the paper is as follows. To enable readers to have a clear and concise understanding,  we provide brief overviews on ExpRK-methods and on the sun-star theory for delay equations in sections \ref{S_exp} and \ref{S_sunstar} respectively. In particular, in section \ref{S_exp} we recall the standard notations of ExpRK-methods and the fundamental results which turns out to be advantageous for the analysis of their application to delay equations. Section \ref{S_sunstar} is devoted to the abstract reformulation of delay equations.   The abstract setting for ExpRK which encompasses both DDEs and REs is presented in section \ref{S_abstract}, which also contains the convergence analysis of exponential Euler method (ExpEuler). Section \ref{S_expDDE} describes the ExpRK-methods for time integration of DDEs and REs, and section \ref{S_nDDE} includes some numerical results. Finally, we draw some conclusions in section \ref{S_end}.

% The title of section 2:

%-----------------------------------------------------------------------------
\section{Explicit ExpRK-methods}\label{S_exp}

In this section, we recall the standard notations and some known results for explicit ExpRK-methods from \cite{HO05a,HO10,LO14}. The convergence analysis carried out therein is based on the abstract framework of analytic semigroups on a Banach space $X$ with norm $\| \cdot \|_X$ and on the following assumptions.

\begin{assumption}\label{A1}
$\mathcal{A}_0:\text{ dom}(\mathcal{A}_0)\subset X\to X$ is the infinitesimal generator of an analytic semigroup $\{\mathcal{T}_0(t)\}_{t \geq 0}$.
\end{assumption}

\begin{assumption}\label{A2}
\eqref{IVPADE} has a sufficiently smooth solution $u:[0,T] \to X$ with derivatives in $X$ and $\mathcal{F}$ is sufficiently often Fr\'echet differentiable in a strip along the exact solution. All occurring derivatives are assumed to be uniformly bounded.
\end{assumption}

The assumption \ref{A1} implies that $\|\mathcal{T}_0(t)\|_{X \leftarrow X} \leq M e^{\omega t}, t \geq 0,$ with $M$ and $\omega$ suitable constants, while from assumption \ref{A2} we have $\mathcal{F}$ is locally Lipschitz continuous in a strip along the exact solution.

Let $t_{n}=nh, n=0,1 \ldots N$ be a uniform mesh with constant step-size $h=\frac{T}{N}.$ A $\nu$-stage explicit ExpRK-method is defined by the formulae

\begin{equation}\label{expmethod}
\setlength\arraycolsep{0.1em}\left\{\begin{array}{rcl}
u_{n+1}&=&\mathcal{T}_0(h)u_n + h\displaystyle\sum_{i=1}^\nu b_i(h\mathcal{A}_0)\mathcal{F}(t_n+c_ih,U_{n,i})\\[2mm]
U_{n,i}&= &\mathcal{T}_0(c_ih)u_n +h\displaystyle\sum_{j=1}^{i-1} a_{ij}(h\mathcal{A}_0)\mathcal{F}(t_n+c_jh,U_{n,j}),\quad i=1,\ldots,\nu
\end{array}
\right.
\end{equation}
where $u_n$ is meant to approximate the value $u(t_n)$ while the internal stages $U_{n,i}$ approximate $u(t_n+c_i h).$ The Butcher tableau has the form
\[
\renewcommand\arraystretch{1.2}
\begin{array}
{c|ccccc|c}
0&0&0&\cdots&0&0&I_{X}\\
c_2 & a_{21}(h\mathcal{A}_0)&0&\cdots&0&0&\mathcal{T}_{0}(c_2h)\\
\vdots&\vdots&\vdots&\ddots&\vdots&\vdots&\vdots\\
c_{\nu} & a_{\nu 1}(h\mathcal{A}_0)&a_{\nu 2}(h\mathcal{A}_0)&\cdots&a_{\nu\nu-1}(h\mathcal{A}_0)&0&\mathcal{T}_{0}(c_\nu h)\\
\hline
& b_{1}(h\mathcal{A}_0) &\cdots &\cdots &b_{\nu-1}(h\mathcal{A}_0)&b_{\nu}(h\mathcal{A}_0)& \mathcal{T}_{0}(h)
\end{array}
\]
The coefficients $b_{i}(h\mathcal{A}_0), a_{ij}(h\mathcal{A}_0)$ are linear combinations of the operators
\begin{equation}\label{phik}
\varphi_k(h\mathcal{A}_0):=\displaystyle\frac{1}{h^k}\int_0^h\mathcal{T}_0(h-s)\frac{s^{k-1}}{(k-1)!}\diff s,\quad k\geq 1,
\end{equation}
which are bounded on $X$ \cite[Lemma 3.1]{HO05a}.  

%\begin{equation}\label{quadF}
%\setlength\arraycolsep{0.1em}\begin{array}{rcl}
%\displaystyle\int_0^h\mathcal{T}_0(h-s)f(s)\diff s&\approx &h\displaystyle\sum_{i=1}^\nu b_i(h\mathcal{A}_0)f(c_ih)\\[2mm]
%\displaystyle\int_0^{c_ih}\mathcal{T}_0(c_ih-s)f(s)\diff s&\approx &h\displaystyle\sum_{j=1}^{i-1} a_{ij}(h\mathcal{A}_0)f(c_jh),\quad i=1,\ldots,\nu
%\end{array}
%\end{equation}
%for nodes $c_i:=\sum_{j=1}^{\nu}a_{ij}(0)$, $i=1,\ldots,\nu$.
%

For an explicit method to converge with the desired stiff order as defined in \cite{HO10}, it is necessary for its coefficients to satisfy certain stiff order conditions. Table \ref{ordercond} collects the conditions up to order 4, expressed in terms of the operators
\begin{equation}\label{psi_ij}
\setlength\arraycolsep{0.1em}\begin{array}{rcl}
\psi_j(h\mathcal{A}_0)&:=&\varphi_j(h\mathcal{A}_0) - \displaystyle\sum_{k=1}^\nu b_k(h\mathcal{A}_0)\frac{c_k^{j-1}}{(j-1)!},j=1,\ldots,\nu\\[2mm]
\psi_{ji}(h\mathcal{A}_0)&:=&\varphi_{i,j}(h\mathcal{A}_0) - \displaystyle\sum_{k=1}^{i-1} a_{ik}(h\mathcal{A}_0)\frac{c_k^{j-1}}{(j-1)!},i,j=1,\ldots,\nu\end{array},
\end{equation}
where $\varphi_{i,j}(h\mathcal{A}_0):=\varphi_j(c_ih\mathcal{A}_0)$.%\begin{table}
%\caption{Order condition for explicit ExpRK methods, where $Z,J$ and $K$ denote arbitrary bounded operators on $X$.}
%\centering
%\begin{tabular}{|c|c|c|}
%\hline
%No.&Order&Order condition\\
%\hline
%$1$ & $1$&$\psi_1(h\mathcal{A}_0)=0$\\
%\hline
%$2$ & $2$&$\psi_2(h\mathcal{A}_0)=0$\\
%$3$ & $2$&$\psi_{1i}(h\mathcal{A}_0)=0,\quad i=1,\ldots,\nu$\\
%\hline
%$4 $& $3$&$\psi_3(h\mathcal{A}_0)=0$\\
%$5$ & $3$&$\displaystyle\sum_{i=1}^\nu b_i(h\mathcal{A}_0)J\psi_{2i}(h\mathcal{A}_0)=0,\quad i=1,\ldots,\nu$\\
%\hline
%$6 $& $4$&$\psi_4(h\mathcal{A}_0)=0$\\
%$7$ & $4$&$\displaystyle\sum_{i=1}^\nu b_i(h\mathcal{A}_0)J\psi_{3i}(h\mathcal{A}_0)=0,\quad i=1,\ldots,\nu$\\
%$7$ & $4$&$\displaystyle\sum_{i=1}^\nu b_i(h\mathcal{A}_0)J\sum_{j=2}^{i-1} a_{ij}(h\mathcal{A}_0)J\psi_{2j}(h\mathcal{A}_0)=0,\quad i=1,\ldots,\nu$\\
%$9$ & $4$&$\displaystyle\sum_{i=1}^\nu b_i(h\mathcal{A}_0)c_iK\psi_{2i}(h\mathcal{A}_0)=0,\quad i=1,\ldots,\nu$\\
%\hline
%\end{tabular}
%\label{ordercond}
%\end{table}

\begin{table}
\caption{Stiff order condition for explicit ExpRK-methods, where $J$ and $K$ denote arbitrary bounded operators on $X$ \cite[Table 2.2]{HO10} and the functions $\psi_{i,j}$ are defined in \eqref{psi_ij}.}
\centering
\begin{tabular}{|c|c|c|}
\hline
No.&Order&Order condition\\
\hline
$1$ & $1$&$\psi_1(h\mathcal{A}_0)=0$\\
\hline
$2$ & $2$&$\psi_2(h\mathcal{A}_0)=0$\\
$3$ & $2$&$\psi_{1i}(h\mathcal{A}_0)=0,\quad i=1,\ldots,\nu$\\
\hline
$4 $& $3$&$\psi_3(h\mathcal{A}_0)=0$\\
$5$ & $3$&$\displaystyle\sum_{i=1}^\nu b_i(h\mathcal{A}_0)J\psi_{2i}(h\mathcal{A}_0)=0,\quad i=1,\ldots,\nu$\\
\hline
$6 $& $4$&$\psi_4(h\mathcal{A}_0)=0$\\
$7$ & $4$&$\displaystyle\sum_{i=1}^\nu b_i(h\mathcal{A}_0)J\psi_{3i}(h\mathcal{A}_0)=0$\\
$8$ & $4$&$\displaystyle\sum_{i=1}^\nu b_i(h\mathcal{A}_0)J\sum_{j=2}^{i-1} a_{ij}(h\mathcal{A}_0)J\psi_{2j}(h\mathcal{A}_0)=0$\\
$9$ & $4$&$\displaystyle\sum_{i=1}^\nu b_i(h\mathcal{A}_0)c_iK\psi_{2i}(h\mathcal{A}_0)=0$\\
\hline
\end{tabular}
\label{ordercond}
\end{table}
One can find the stiff order conditions for explicit ExpRK-methods up to order 5 in \cite[Table 1]{LO14} and order 6 in \cite[Table 1]{LA24}. The following theorem determines (a lower bound for) the convergence rate of explicit ExpRK-methods for constant step sizes that only requires a weakened form of the conditions of order $p.$ 

\begin{theorem}\label{th_convergence}
%Assume that $\|\mathcal{T}_0(t)\|_{t\in[0,T]}$ is uniformly bounded and that the function $t\mapsto\mathcal{F}(t,u(t))$ is $p$ times Fr\'echet-differentiable with bounded derivatives, for $u$ a sufficiently smooth solution of \eqref{IVPADE}. 
Let the initial value problem \eqref{IVPADE} satisfy assumptions \ref{A1}-\ref{A2}. Consider an explicit ExpRK-method that satisfies the order conditions up to order $p-1,$ as well as $\displaystyle\sum_{i=1}^{\nu}b_i(0)c_i^{p-1}=\frac{1}{p}$ and the conditions of order $p$ in the \emph{weak form}, i.e., with $b_i(h\mathcal{A}_0)$ replaced by $b_i(0)$ for $i=2,\ldots,\nu.$  Then, the method is convergent of order  $p$ and the error bound
\begin{equation}\label{errbound}
\|u_n-u(t_n)\|_X\leq Ch^p
\end{equation}
holds uniformly for $nh\in\displaystyle\left[0,T\right]$ for a constant $C$ that depends on $T,$ but is independent of $n$ and $h$.

\end{theorem}

In the next sections we will consider the following methods up to order 3. The first is the exponential Euler method (ExpEuler), having stiff order 1 and Butcher tableau \begin{equation}\label{eulerButcher}
\renewcommand\arraystretch{1.2}
\begin{array}
{c|c|c}
0 & 0&I\\
\hline
 & \varphi_1&\varphi_1.
\end{array}
\end{equation}
The Heun method, having stiff order 2, is defined by
\begin{equation}\label{heunButcher}
\renewcommand\arraystretch{1.2}
\begin{array}
{c|cc|c}
0 & 0&0&I\\
c_2&c_2\varphi_{1,2}&0&\varphi_{1,2}\\
\hline
 & \varphi_1-\frac{1}{c_2}\varphi_2&\frac{1}{c_2}\varphi_2&\varphi_1\\
\end{array}.
\end{equation}
for $c_2=1$. Note that there is also a version of \eqref{heunButcher} which only satisfies the conditions of order 2 in the weak form defined as in Theorem \ref{th_convergence}, which yields the following family of methods
\begin{equation*}
\renewcommand\arraystretch{1.2}
\begin{array}
{c|cc|c}
0 & 0&0&I\\
c_2&c_2\varphi_{1,2}&0&\varphi_{1,2}\\
\hline
 & \left(1-\frac{1}{2c_2}\right)\varphi_1&\frac{1}{2c_2}\varphi_1&\varphi_1\\
\end{array}.
\end{equation*}
Finally, we consider the following method of stiff order 3:
\begin{equation}\label{order3Butcher}
\renewcommand\arraystretch{1.2}
\begin{array}
{c|ccc|c}
0 & 0&0&0&I\\
\frac{1}{2}&\frac{1}{2}\varphi_{1,2}&0&0&\varphi_{1,2}\\
\frac{2}{3}&\frac{2}{3}\varphi_{1,3}-\frac{8}{9}\varphi_{2,3}&\frac{8}{9}\varphi_{2,3}&0&\varphi_{1,3}\\
\hline
 & \varphi_1-\frac{3}{2}\varphi_2&0&\frac{3}{2}\varphi_2&\varphi_1\\
\end{array}
\end{equation}
Note that \eqref{order3Butcher} only satisfies the conditions of order 3 in the weak form.
% The title of your section 3:
%-----------------------------------------------------------------------------
\section{Sun-star theory for delay equations}\label{S_sunstar}

In this section we recall the abstract formulation of delay equations by using the sun-star theory and we provide the basics for our abstract setting. The interested readers are referred to \cite{twin,dgg07,diekmann95} for insights. Hereafter $d$ is a positive integer and $| \cdot |$ denotes a vector norm in $\mathbb{R}^{d}.$

\subsection{Delay differential equations}\label{SS_aDDE}
A delay differential equation (DDE) with finite delay $\tau\in(0,+\infty)$ can be written as
\begin{equation}\label{eq:DDE}
\setlength\arraycolsep{0.1em}\begin{array}{rcl}
x'(t)&=&F(t,x_{t}), \ t \geq 0,
\end{array}
\end{equation}
where $x_{t}$ defined as
\begin{equation*}
x_{t}(\theta):=x(t+\theta),\quad\theta\in[-\tau,0],
\end{equation*}
is the \emph{history function} at time $t$, which lives on the \emph{history space} $X,$ and $F:[0,+\infty)\times X\rightarrow\mathbb{R}^{d}$ is generally a sufficiently smooth nonlinear map. $X$ is classically defined as the Banach space $X:=C([-\tau,0],\mathbb{R}^{d}),$ equipped with the uniform norm $\| \phi\|_X:=\max_{\theta\in[-\tau,0]}|\phi(\theta)|, \ \phi \in X$ (see, e.g., \cite{hale77}).
 
An initial value problem for \eqref{eq:DDE} can then be formulated as
\begin{equation}\label{IVPDDE}
\left\{\setlength\arraycolsep{0.1em}\begin{array}{rcll}
x'(t)&=&F(t,x_{t}),&\quad t \in [0,T], \\[1mm]
x_0(\theta) &=&\phi(\theta),&\quad \theta\in[-\tau,0],
\end{array}
\right.
\end{equation} 
for a given function $\phi \in X$.
\begin{remark}\label{breakingpoints}

Even if $F,\phi$ are smooth functions, when $\phi'(0) \neq F(0,\phi)$ the solution $x$ of \eqref{IVPDDE} does not link smoothly to the initial function at $0$, the higher-order derivatives of $x$ along the integration interval $[0,T]$ present discontinuities points, often called \emph{breaking points}. However, due to the \emph{smoothing effect} of DDEs the solution attains more regularity from one breaking point to the next.
\end{remark} 
The DDE \eqref{eq:DDE} can be analyzed in the general framework provided by the perturbation theory for dual semigroups \cite{cdght87}, as done in \cite{diekmann95}. The dual space of $X$ can be identified with the set $X^*=\text{NBV}([0,\tau],\R^d)$ of functions of normalized bounded variation. If $\{T_0(t)\}_{t\geq 0}$ is the \emph{shift} $\mathcal{C}_0$-semigroup given by
\begin{equation}\label{T0}T_0(t)\phi=
\left\{\setlength\arraycolsep{0.1em}\begin{array}{rcl}
&\phi(t),&\quad t\in[-\tau,0], \\[1mm]
&\phi(0),&\quad t\geq 0,
\end{array}
\right.
\end{equation}
then, under the same representation, its adjoint operator is defined by
\begin{equation*}
(T_0^*(t)f)(\theta)=f(t+\theta),
\end{equation*}
and the largest subspace of $X^*$ on which $\{T_0^*(t)\}_{t\geq 0}$ is strongly continuous can be identified with $X^{\odot}= \R^d\times L^1([0,\tau],\R^d)$. The idea is to work with embedding $j:X\to X^{\odot *}=\R^d\times L^{\infty}([-\tau,0],\R^d)$ defined by
\begin{equation*}
j\phi(\psi^{\odot}) = \psi^{\odot}(\phi),\quad \phi\in X,\,\psi^{\odot}\in X^{\odot},
\end{equation*}
which, in our case, reads
\begin{equation*}
j\phi=(\phi(0);\phi),\quad \phi\in X.
\end{equation*}
Moreover,
\begin{equation}\label{T0sunstar}T_0^{\odot *}(t)(\alpha;\phi)=
\left(\alpha;\theta\mapsto\left\{\setlength\arraycolsep{0.1em}\begin{array}{rcl}
&\alpha,&\quad \theta\in[-t,0], \\[1mm]
&\phi(t+\theta),&\quad \theta\in[-\tau,-t)
\end{array}
\right.\right),
\end{equation}
confirming the identity $T_0^{\odot *}(j\phi)= j(T_0(t)\phi)$.
Finally, it can be proved that $X$ is \emph{sun-reflexive}, i.e., 
\begin{equation*}
X^{\odot\odot}=\{(\alpha;\phi)\in X^{\odot *}\lvert\phi(0)=\alpha\}=j(X).
\end{equation*}

It is the possibility of identifying $X^{\odot\odot}$ with $X$ that allows us to recast \eqref{eq:DDE} as a semilinear abstract differential equation and to derive an equivalent abstract variation-of-constant equation, which in turn enable us to apply ExpRK-methods described in Section \ref{S_exp}.

By setting
\begin{equation}\label{u_DDE}
u(t):=jx_t = (x(t);x_t),\quad t\geq 0,
\end{equation}
the infinitesimal generator $\mathcal{A}_0$ of $\{\mathcal{T}_0(t):=T_0^{\odot *}(t)\}_{t\geq 0}$ is given by
\begin{equation}\label{A_DDE}
\mathcal{A}_0(\alpha;\psi)=(0;\psi'),\quad D(\mathcal{A}_0)=\{(\alpha;\psi)\in X\lvert \psi\text{ is Lipschitz-continuous, }\psi(0)=\alpha\}
\end{equation}
and $\{\mathcal{T}_0(t)\}$ is strongly continuous semigroup on $X\cong X^{\odot\odot}$. For $\mathcal{F}$ defined by
\begin{equation}\label{F_DDE}
\mathcal{F}(t,(\psi(0);\psi))=(F(t,\psi);0),
\end{equation}
and $u_0=\psi$ the \eqref{IVPDDE} can be formally written as \eqref{IVPADE}. Note that $\mathcal{F}$ inherits the regularity of $F$.  
By formally integrating \eqref{IVPADE} we obtain the variation-of-constants equation
\begin{equation}\label{AVCF}
u(t)=\mathcal{T}_0(t)u_0+\int_0^t \mathcal{T}_0(t-s)\mathcal{F}(s,u(s))\diff s,
\end{equation}
whose interpretation requires to define the weak$^\star$ integral \cite[pag.52]{diekmann95}. The equivalence between \eqref{AVCF} and \eqref{IVPDDE} is guaranteed by the following result.

\begin{theorem}[{\cite[Theorem III.4.1]{diekmann95}}]
If $x:[-\tau,T]\to\R^d$ is a continuous solution of \eqref{IVPDDE}, then the function $u:t\mapsto jx_t$ is a solution of \eqref{AVCF} for $u_0=(\phi(0);\phi)$ which is continuous in $[0,T]$.
Viceversa, if $u:[0,T]\to X$ is a continuous solution of \eqref{AVCF} for $u_0=(\phi(0);\phi)$, then the function
\begin{equation*}
x(t):=
\left\{\setlength\arraycolsep{0.1em}\begin{array}{rcl}
&\phi(t),&\quad t\in[-\tau,0], \\[1mm]
&(j^{-1}u)(0),&\quad t\in[0,T]
\end{array}
\right.
\end{equation*}
is a solution of \eqref{IVPDDE} which is continuous in $[-\tau,T]$.
\end{theorem}

\subsection{Renewal equations}\label{SS_aRE}

The perturbation theory for dual semigroups \cite{cdght87} we have resorted to in order to reformulate \eqref{eq:DDE}, can be extended to Renewal Equations (RE), i.e., equations of the form
\begin{equation}\label{eq:RE}
\setlength\arraycolsep{0.1em}\begin{array}{rcl}
x(t)&=&F(t,x_{t}),
\end{array}
\end{equation}
as done in \cite{dgg07}. The main difference with respect to the treatment of DDEs lies in the choice of the state space, this time $X:=L^1([-\tau,0],\R^d)$. An IVP for \eqref{eq:RE} can then be formulated as
\begin{equation}\label{IVPRE}
\left\{\setlength\arraycolsep{0.1em}\begin{array}{rcll}
x(t)&=&F(t,x_{t}),&\quad t>0, \\[1mm]
x_0(\theta) &=&\phi(\theta),&\quad \theta\in[-\tau,0],
\end{array}
\right.
\end{equation} 
for a given function $\phi\in X$. The $\mathcal{C}_0$-semigroup $\{T_0(t)\}_t\geq 0$  corresponding to \eqref{eq:RE} is given by
\begin{equation}\label{T0RE}
(T_0(t)\phi)(\theta)=
\left\{\setlength\arraycolsep{0.1em}\begin{array}{rcl}
&\phi(t+\theta),&\quad t+\theta\in[-\tau,0], \\[1mm]
&0,&\quad t+\theta> 0
\end{array}
\right.
\end{equation}
for $\theta\in[-\tau,0]$. The dual space $X^*$ can be represented by $L^{\infty}([0,\tau],\R^d)$ and, under this representation, the adjoint operator of \eqref{T0RE} is defined by
\begin{equation*}
(T_0^*(t)f)(\theta)=
\left\{\setlength\arraycolsep{0.1em}\begin{array}{rcl}
&f(t+\theta),&\quad t+\theta\in[0,\tau], \\[1mm]
&0,&\quad t+\theta> \tau.
\end{array}
\right.
\end{equation*}
Since translation is not continuous in $X^*$, we have 
\begin{equation*}
X^{\odot}=C_0([0,\tau),\R^d):=\left\{f\in C([0,\tau),\R^d)\middle| \lim_{t\to\tau}f(t)=0\right\}
\end{equation*}
and, thus, $X^{\odot*}=\text{NBV}((-\tau,0],\R^d)$. The corresponding embedding $j:X\to X^{\odot*}$ reads
\begin{equation}\label{j_RE}
(j\phi)(\theta)=\int_0^{\theta}\phi(s)\diff s,\qquad \theta\in(-\tau,0],
\end{equation}
while
\begin{equation}\label{T0REsunstar}(\mathcal{T}_0(t)\phi)(\theta)=(T_0^{\odot *}(t)\phi)(\theta)=
\left\{\setlength\arraycolsep{0.1em}\begin{array}{rcl}
&0,&\quad \theta\in[-t,0], \\[1mm]
&\phi(t+\theta),&\quad \theta\in[-\tau,-t)
\end{array}
\right.
\end{equation}
is strongly continuous on
\begin{equation*}
X^{\odot\odot}=\{f\in X^{\odot*}|f\in AC((-\tau,0],\R^d)\}=j(X)
\end{equation*}
and has infinitesimal generator 
\begin{equation}\label{A_RE}
\mathcal{A}_0(\phi)=\phi',\quad D(\mathcal{A}_0)=\{\phi\in X\lvert \phi'\in X,\,\phi(0)=0\}.
\end{equation}
Arguing as we did in Section \ref{SS_aDDE} with DDEs, \eqref{eq:RE} can be recast into an abstract differential equation for $u(t):=jx_t$ for $j$ in \eqref{j_RE} and
\begin{equation}\label{F_RE}
\mathcal{F}(t,j\phi):=F(t,\phi)H,\qquad H(\theta):=\left\{\setlength\arraycolsep{0.1em}\begin{array}{rcl}
&1,&\quad \theta\in[-\tau,0), \\[1mm]
&0,&\quad \theta=0.
\end{array}
\right.
\end{equation}
Also for REs we can write the integral abstract equation \eqref{AVCF}, interpreted in weak$^{\star}$ form, and the equivalence between \eqref{AVCF} and \eqref{IVPRE} is guaranteed by the following result.
\begin{theorem}[{\cite[Theorem 3.7]{dgg07}}]
If $x:[-\tau,T]\to\R^d$ is an $L^1$ solution of \eqref{IVPRE}, then the function $\eta:t\mapsto x_t$ is continuous in $[0,T]$ and $u=j\eta$ is a solution of \eqref{AVCF} for $u_0=j\phi$.
Viceversa, if $u=j\eta$ is a solution of \eqref{AVCF} with $u_0=j\phi$ for some continuous $\eta:[0,T]\to X$, then the function
\begin{equation*}
x(t):=
\left\{\setlength\arraycolsep{0.1em}\begin{array}{rcl}
&\phi(t),&\quad t\in[-\tau,0], \\[1mm]
&F(t,\eta(t)),&\quad t\in[0,T]
\end{array}
\right.
\end{equation*}
is a solution of \eqref{IVPRE} which belongs to $L^1([-\tau,T],\R^d)$.
\end{theorem}
%-----------------------------------------------------------------------------
%-----------------------------------------------------------------------------
\section{The abstract setting}\label{S_abstract}

In this section, based on the sun-star theory in the sun-reflexivity case, we introduce the abstract setting which enables to encompass the abstract formulations of DDEs and REs, and to analyse the convergence of ExpRK methods in a unifying way.  We pay particular attention to the weak$^\star$ integral. 

Starting from the Banach space of functions $X$ and the $\mathcal{C}_0$-semigroup $\{T_0(t)\}_{t \geq 0}$, the space $X^{\odot} \subset X^\star$ has been initially introduced to then arrive to the space $X^{\odot\star},$ in which $X$ is embedded, and to the extension $\mathcal{T}_0(t)$ to $X^{\odot\star}$ of $T_0(t)$. Through isometric isomorphisms, the spaces $X^{\odot}, X^{\odot\star}$ can be identified with spaces of functions, i.e. $X^{\odot} \cong Z$ and $X^{\odot\star} \cong Y.$ For delay equations $T_0(t)$ is the shift semigroup on $X$ and $\| \mathcal{T}_0(t)\|_{Y\leftarrow Y} \leq 1, \ t \geq 0.$ However, for a more general applicability, we assume that $\| \mathcal{T}_0(t)\|_{Y\leftarrow Y} \leq  M e^{\omega t}, \ t \geq 0,$ for some $\omega.$

Let  $Y$ be a Banach space of functions with values in $\R^d$, $X \subset Y$ be a subspace of functions with higher regularity. We consider the abstract Cauchy problem \eqref{IVPADE}. Given $u_0 \in X$, we are interested in a \emph{mild} solution of \eqref{IVPADE}, that is, a solution $u$ of the abstract integral equation \eqref{AVCF}. 
We make the following assumptions. 

\begin{assumption}\label{AA1}
$\mathcal{A}_0:D(\mathcal{A}_0) \subset X \to Y$ is the infinitesimal generator of a strongly continuous semigroup $\{\mathcal{T}_0(t)\}_{t\geq0}$ on $X$ with
\begin{equation*}
D(\mathcal{A}_0)\subset\{\psi\in Y \lvert \mathcal{A}_0\psi\in Y\}.
\end{equation*}
\end{assumption} 

\begin{assumption}\label{AA2}
 $\mathcal{F}$ is locally Lipschitz-continuous with respect to its second argument in a strip along the exact solution $u$, i.e., for all $R\in[0,\overline{R}]$ there exists $L>0$ such that
\begin{equation*}
\|\mathcal{F}(t,w)-\mathcal{F}(t,v)\|_Y\leq L\|w-v\|_Y,\quad t\in[0,T],\quad\|w-u\|_Y,\|v-u\|_Y\leq R.
\end{equation*}

\end{assumption}
\begin{assumption}\label{AA3}
$u:[0,T] \to X$ is a sufficiently smooth solution and $\mathcal{F}$ is
is sufficiently often Fr\'echet differentiable in a strip along the exact solution. All occurring derivatives are assumed to be uniformly bounded.
\end{assumption} 
Since $\{\mathcal{T}_0(t)\}_{t\geq0}$ is only strongly continuous on $X$ and not necessarily on $Y$, the integral in \eqref{AVCF} needs to be interpreted in the weak$^*$ sense. This means that an integral of the form $Q:=\displaystyle \int_0^t q(s)\diff s$ is defined as the unique functional mapping $z\in Z$ to $\int_0^t q(s)z\diff s$, that is
\begin{equation*}
\langle Q, z\rangle = \int_0^t\langle q(s),z\rangle\diff s,
\end{equation*}
where $\langle y,z\rangle$ denotes the ``concrete" pairing \cite[page 41]{diekmann95} between the functions $y \in Y, z \in Z$. Note that $|\langle y,z\rangle | \leq \|y\|_Y \|z\|_Z, \ y \in Y, z \in Z.$
\begin{lemma}\label{phibounded}
The operators \eqref{phik} are norm bounded.
\end{lemma}

\begin{proof} 
From \cite[Lemma 2.3]{diekmann95} we have that if the function $f:[0,T] \to Y$ is norm continuous, then the function $v:[0,T] \to Y$ defined as weak$^*$ integral $v(t)=\int_0^t \mathcal{T}_0(t-s) f(s) ds, \ t \in [0,T],$ is norm continuous, takes values in $X$ and $$\|v(t)\|_Y \leq  M  \frac{e^{\omega t}-1}{\omega} \sup_{0 \leq s \leq t} \| f(s) \|_Y, t \in [0,T].$$  Then, the boundedness of \eqref{phik} easily follows from
$$\left\| \int_0^h\mathcal{T}_0(h-s)\frac{s^{k-1}}{(k-1)!}\diff s \right\|_Y \leq M \frac{e^{\omega h}-1}{\omega}\frac{h^{k-1}}{(k-1)!}$$ for $k\geq 1.$
\end{proof}
Since the coefficients of the explicit ExpRK methods \eqref{expmethod} are linear combinations of \eqref{phik}, from Lemma \ref{phibounded}, we get that they are the bounded operators on $Y.$ 

For a better understanding of the weak$^\star$ integral, we first study the convergence of the ExpEuler method, having Butcher tableau \ref{eulerButcher}.

Let $f(t)=\mathcal{F}(t,u(t)), t \in [0,T].$  We have
\begin{equation*}
\setlength\arraycolsep{0.1em}\begin{array}{rcll}
u(t_{n+1})&=&\displaystyle\mathcal{T}_0(h)u(t_n)+\int_0^h\mathcal{T}_0(h-s)f(t_n+s)\diff s \\[1mm]
 &=&\displaystyle\mathcal{T}_0(h)u(t_n)+\sigma_{n+1}\end{array}
\end{equation*}
where the local error is given by 
$$\sigma_{n+1}=\int_0^h\mathcal{T}_0(h-s)(f(t_n+s)-f(t_n))\diff s=\int_0^h\mathcal{T}_0(h-s)\int_0^sf'(t_n+\sigma)\diff \sigma \diff s.$$
\begin{lemma}\label{lbounded}
Assume that $f' \in L^{\infty}((0,T),Y).$ Then
 $$
 \left\|\sum_{k=0}^n \mathcal{T}_0(kh)\sigma_{n+1-k}\right\|_Y \leq C h^2 \sup_{0 \leq T \leq T} \| f'(t) \|_Y
 $$
 holds with a constant $C$, uniformly for $0 \leq t_n \leq T.$
 \end{lemma}

\begin{proof} 

From \cite[Lemma 3.14, Lemma 2.3]{diekmann95} we get 
$$\mathcal{T}_0(kh)\sigma_{n+1-k}=\int_0^h\mathcal{T}_0((k+1)h-s)\int_0^sf'(t_n+\sigma)\diff \sigma \diff s,$$
and
$$\left\|\int_0^h\mathcal{T}_0((k+1)h-s)\int_0^sf'(t_n+\sigma)\diff \sigma \diff s\| \leq \frac{h^2}{2} M  \frac{e^{\omega (k+1)h}-1}{\omega} \sup_{0 \leq \sigma \leq h} \| f'(t_n+\sigma) \right\|_Y,$$
for $k=0,...,n.$ Thanks to the discrete Gronwall Lemma, we get the thesis with $C:=\frac{1}{2} M \frac{e^{\omega T}-1}{\omega}.$\end{proof}
\noindent 
For the ExpEuler method we have the following convergence result

\begin{theorem}
Assume \ref{AA1}-\ref{AA2} and $f(t):=F(t,u(t)), t \in [0,T]$ is such that $f' \in L^{\infty}((0,T),Y).$ Then the error bound 
$$
|| u(t_n)-u_n ||_X \leq \tilde C \max_{ 0 \leq t \leq T} \| f'(t) \|_Y  \cdot h,
$$
holds for $n=1,\ldots,N.$ The constant $ \tilde C$ depends on $T$ but it is independent of $h$ and $n.$
\end{theorem}
\begin{proof}

The error $E_{n}:=u(t_n)-u_n$ of the ExpEuler method satisfies the recursion
\begin{equation}\label{errEul}
\begin{array}{ll}
E_{n+1}&= \mathcal{T}_0(h)E_{n}+h\varphi_1(h\mathcal{A}_0)\Delta_{n+1}+\sigma_{n+1}, \ n=0,\ldots,N-1,
\end {array}
\end{equation}
with $\Delta_{n+1}:=\mathcal{F}(t_n,u(t_n))-\mathcal{F}(t_{n},u_{n}), \ n=0,\dots,N-1.$ Solving it yields
$$
E_{n+1}=\sum\limits_{k=0}^n \mathcal{T}_0(kh)\varphi_1(h\mathcal{A}_0)\Delta_{n+1-k}+ \sum\limits_{k=0}^n \mathcal{T}_0(kh)\sigma_{n+1-k}.
$$
Now $\Delta_{0}=0$ and, from \cite[Lemma 3.14,Lemma 2.3]{diekmann95} and the local Lipschitz-continuity of  $\mathcal{F}$, we get
$$
\mathcal{T}_0(kh)h \varphi_1(h\mathcal{A}_0)\Delta_{n+1-k}=\int\limits_0^{h} \mathcal{T}_0((k+1)h-s) ds \ \Delta_{n+1-k}, \ k=1,\ldots,n,
$$
and 
$$
||\mathcal{T}_0(kh)h \varphi_1(h\mathcal{A}_0)\Delta_{n+1-k}||_Y \leq h C ||E_{k}||_Y, \ k=1,\ldots,n.
$$
for a constant $C$.

From Lemma \ref{lbounded} and using the bounds given above we have for $n=0,\ldots,N-1$ and 

$$
\begin{array}{ll}
||E_{n+1}||_X&\leq h C \sum\limits_{k=1}^n ||E_{k}||_Y+ h C\sup_{ 0 \leq t  \leq T} \| f'(t) \|_Y \\
\end{array}
$$
and from the discrete Gronwall Lemma the thesis follows.
\end{proof}

In this abstract setting, under the assumptions \ref{AA1}-\ref{AA2}-\ref{AA3}, the analysis of the convergence of a scheme that fulfills the order condition up to order $p$ can be carried out as in \cite{HO10,LO14} and a theorem analogous to Theorem \ref{th_convergence} holds.
%-----------------------------------------------------------------------------
%-----------------------------------------------------------------------------
\section{ExpRK-methods for delay equations}\label{S_expDDE}

We are interested in applying explicit ExpRK-methods to integrate delay equations.\\

Let us start from DDEs, dealt with in section \ref{SS_aDDE}. The goal is to integrate \eqref{IVPADE} in time for $u$ defined in \eqref{u_DDE}, $\mathcal{A}_0$ as in \eqref{A_DDE} and $\mathcal{F}$ as in \eqref{F_DDE}. From \eqref{phik}, \eqref{T0sunstar} and 
$$
\int_0^h\mathcal{T}_0(h-s)(f(s);0)\diff s =\left(\int_0^hf(s)\diff s;\int_0^{\max\{0,h+\cdot\}}f(s)\diff s\right)
$$
we get
\begin{equation}\label{phik_DDE}
h^k\varphi_k(h\mathcal{A}_0)=\int_0^h\mathcal{T}_0(h-s)\left(\frac{s^{k-1}}{(k-1)!};0\right)\diff s=\left(\frac{h^k}{k!};\frac{\max\{0,h+\cdot\}}{k!}\right),\quad k\geq 1
\end{equation}
Since the coefficients $a_{ij}(h\mathcal{A}_0)$ and $b_i(h\mathcal{A}_0)$ are linear combinations of $\varphi_k(h\mathcal{A}_0)$, \eqref{phik_DDE} allows us to write \eqref{expmethod} explicitly for a given method. For instance, the ExpEuler method \eqref{eulerButcher} is given for $u_n=(y_n;\eta_n)$ by
\begin{equation*}
\left\{\setlength\arraycolsep{0.1em}\begin{array}{rll}
y_{n+1}&=&y_n+hF(t_n,\eta_n)\\[1mm]
\eta_{n+1}(\theta)&=&\left\{\setlength\arraycolsep{0.1em}\begin{array}{rll}
&\eta_n(h+\theta),&\quad \theta\in[-\tau,-h), \\[1mm]
&y_n+(h+\theta)F(t_n,\eta_n),&\quad \theta\in[-h,0].
\end{array}
\right.
\end{array}
\right.
\end{equation*}
Similarly, the exponential Heun method, having Butcher tableau \eqref{heunButcher} and stiff order two, is given by
\begin{equation*}
\left\{\setlength\arraycolsep{0.1em}\begin{array}{rll}
y_{n+1}&=&y_n+\displaystyle\frac{h}{2}F(t_n,\eta_n)+\frac{h}{2}F(t_{n+1},\eta_{2,n})\\[1mm]
\eta_{n+1}(\theta)&=&\left\{\setlength\arraycolsep{0.1em}\begin{array}{rll}
&\eta_n(h+\theta),&\quad \theta\in[-\tau,-h), \\[1mm]
&y_n+\displaystyle\left(h+\theta-\frac{(h+\theta)^2}{2h}\right)F(t_n,\eta_n)+\frac{(h+\theta)^2}{2h} F(t_{n+1},\eta_{2,n}),&\quad \theta\in[-h,0],
\end{array}
\right.\\[1mm]
\eta_{2,n}(\theta)&=&\left\{\setlength\arraycolsep{0.1em}\begin{array}{rll}
&\eta_n(h+\theta),&\quad \theta\in[-\tau,-h), \\[1mm]
&y_n+\displaystyle(h+\theta)F(t_n,\eta_n),&\quad \theta\in[-h,0].
\end{array}
\right.
\end{array}
\right.
\end{equation*}
As a third example, we consider the method of stiff order three \cite[equation (5.8)]{HO05a} after setting $c_2=\frac{1}{2}$, having Butcher tableau \eqref{order3Butcher}. For DDEs, it reads
\begin{equation*}
\left\{\setlength\arraycolsep{0.1em}\begin{array}{rll}
y_{n+1}&=&y_n+\displaystyle\frac{h}{4}F(t_n,\eta_n)+\frac{3}{4}hF\left(t_n+\frac{2}{3}h,\eta_{3,n}\right)\\[1mm]
\eta_{n+1}(\theta)&=&\left\{\setlength\arraycolsep{0.1em}\begin{array}{rll}
&\eta_n(h+\theta),&\quad \theta\in[-\tau,-h), \\[1mm]
&y_n+\displaystyle\left(h+\theta-\frac{3(h+\theta)^2}{2h}\right)F(t_n,\eta_n)+\frac{3(h+\theta)^2}{4h} F\left(t_n+\frac{2}{3}h,\eta_{3,n}\right),&\quad \theta\in[-h,0],
\end{array}
\right.\\[1mm]
\eta_{3,n}(\theta)&=&\left\{\setlength\arraycolsep{0.1em}\begin{array}{rll}
&\eta_n(\frac{2}{3}h+\theta),&\quad \theta\in[-\tau,-\frac{2}{3}h), \\[1mm]
&y_n+\displaystyle\left(\frac{2}{3}h+\theta-\frac{(\frac{2}{3}h+\theta)^2}{h}\right)F(t_n,\eta_n)+\frac{(\frac{2}{3}h+\theta)^2}{h}F\left(t_n+\frac{1}{2}h,\eta_{2,n}\right),&\quad \theta\in[-\frac{2}{3}h,0].
\end{array}
\right.\\[1mm]
\eta_{2,n}(\theta)&=&\left\{\setlength\arraycolsep{0.1em}\begin{array}{rll}
&\eta_n(\frac{1}{2}h+\theta),&\quad \theta\in[-\tau,-\frac{1}{2}h), \\[1mm]
&y_n+\displaystyle\left(\frac{1}{2}h+\theta\right)F(t_n,\eta_n),&\quad \theta\in[-\frac{1}{2}h,0].
\end{array}
\right.
\end{array}
\right.
\end{equation*}
As observed in section \ref{S_exp}, the methods \eqref{eulerButcher} and \eqref{heunButcher} satisfy the order conditions in Table \ref{ordercond} up to order 1 and 2, respectively. On the other hand, \eqref{order3Butcher} only satisfies those of order 3 in the weak sense, defined as in Theorem \ref{th_convergence}. By the same theorem, this is sufficient to obtain the convergence of order 3 when using a constant stepsize under assumption \ref{AA3}.

\begin{remark}\label{distrdel}In the presence of breaking points (see Remark \ref{breakingpoints}), these can be included into the mesh by using variable stepsizes $h_n=t_{n+1}-t_n$ so that $x(t)$, $t \in [t_n,t_{n+1}]$ is sufficiently smooth. Recall that $x_t$ attains more regularity as $t$ increases, thanks to the smoothing effect (Remark \ref{breakingpoints}). In this case, the strong order conditions of order $p$ ensure $p$-order of convergence.\\
In case of DDEs with distributed delays, the evaluation of $F$ involves an integral, so, in general, one constructs an approximation $\tilde F$ by using a suitable composite quadrature rule (see, e.g., \cite{brunner1986numerical,Vermiglio1988,Vermiglio1992}) or available codes for numerical integration. In this respect we can consider a formula that guarantees at least the same order of the \emph{main} error, that is, the one that we would obtain if $F$ could be computed exactly. Alternatively, given a tolerance TOL which bounds the quadrature error from above, we can accept the fact that the final error decays down to TOL and not to 0.
\end{remark}

\begin{remark}
The ExpRK methods in this abstract setting can be easily extended to semilinear DDEs of this form 
\begin{equation}\label{SLDDE}
\left\{
\begin{aligned}
&{x}'(t)=L x(t)+G(t,x_t), t \in [0,T], \\
&x_0=\phi,
\end{aligned}
\right.
\end{equation}
where $L$ is a stiff matrix, i.e., possessing eigenvalues with large negative real parts. Indeed by defining 
\begin{equation}\label{A_SLDDE}
\mathcal{A}_0(\alpha;\phi)=(L\alpha;\phi'),\quad D(\mathcal{A}_0)=\{(\alpha;\phi)\in X\lvert \ \phi\text{ is Lipschitz-continuous },\phi(0)=\alpha\},
\end{equation}
\begin{equation}\label{G_SLDDE}
\mathcal{F}(t,(\phi(0);\phi))=(G(t,\phi);0),
\end{equation}
and $u_0=(\phi(0);\phi),$ \eqref{SLDDE} can be formally written as \eqref{IVPADE}. We have that $\mathcal{A}_0$ is the infinitesimal generator of the semigroup $\{\mathcal{T}_0(t)\}_{t \geq 0}$ with
\begin{equation}\label{T0SLDDE}\mathcal{T}_0(t)(\alpha;\phi)=
\left(e^{Lt}\alpha;\theta\mapsto\left\{\setlength\arraycolsep{0.1em}\begin{array}{rcl}
&e^{L(t+\theta)}\alpha,&\quad \theta\in[-t,0], \\[1mm]
&\phi(t+\theta),&\quad \theta\in[-\tau,-t).
\end{array}
\right.\right).
\end{equation}
For instance, the ExpEuler method \eqref{eulerButcher} reads as
\begin{equation*}
\left\{\setlength\arraycolsep{0.1em}\begin{array}{lll}
y_{n+1}&=&e^{Lh}y_n+h\varphi_1(hL)G(t_n,\eta_n),\\[1mm]
\eta_{n+1}(\theta)&=&\left\{\setlength\arraycolsep{0.1em}\begin{array}{rll}
&\eta_n(h+\theta),&\quad \theta\in[-\tau,-h), \\[1mm]
&e^{L(h+\theta)}y_n+(h+\theta)\varphi_1((h+\theta)L)G(t_n,\eta_n),&\quad \theta\in[-h,0].
\end{array}
\right.
\end{array}
\right.
\end{equation*}
and its implementation requires an efficient algorithm to compute the product of a exponential matrix and a vector (see for instance \cite{AH11}).
Thus the application of ExpRK methods in the sun-star abstract setting leads to the class of ExpRK for \eqref{SLDDE}, showing the flexibility of the approach. 

Explicit ExpRK schemes for semilinear DDEs with a single discrete delay have been considered in \cite{FZ21,ZZO16}. The former focuses on order conditions and convergence analysis in the autonomous case, while the latter explores the stability properties of explicit ExpRK schemes as well as Magnus integrators equipped with Lagrangian interpolation.
\end{remark}

Finally, we underline that ExpRK methods for DDEs reduce to the class of functional continuous RK schemes in \cite{BGMZ09},  where one can find the convergence analysis and explicit schemes up to order four, based on the results in \cite{MTV05}. The definitions of uniform order, discrete order of continuous RK correspond to the definitions of stiff order and weak order conditions of  ExpRK methods for DDEs. Moreover, methods of order 5 and 6 are available from \cite{LA24,LO14}. \\

In order to define ExpRK-methods for REs explicitly, we consider \eqref{IVPADE} for $u$ defined as $u(t):=jx_t$ with $j$ in \eqref{j_RE}, $\mathcal{A}_0$ in \eqref{A_RE} and $\mathcal{F}$ in \eqref{F_RE}. By \eqref{phik}, \eqref{T0REsunstar} and
$$
\int_0^h\mathcal{T}_0(h-s)f(s)H\diff s=\int_{\max\{0,h+\cdot\}}^tf(s)\diff s
$$
 we have
\begin{equation}\label{phik_RE}
h^k\varphi_k(h\mathcal{A}_0)=\int_0^h\mathcal{T}_0(h-s)\frac{s^{k-1}}{(k-1)!}\diff s=\frac{h^k-\max\{0,h+\cdot\}^k}{k!},\quad k\geq 1.
\end{equation}
The exponential Euler method \eqref{eulerButcher} is given for $u_n=j\eta_n$ by
\begin{equation*}
\left\{\setlength\arraycolsep{0.1em}\begin{array}{rll}
u_{n+1}(\theta)&=&\left\{\setlength\arraycolsep{0.1em}\begin{array}{rll}
&u_n(h+\theta)+hF(t_n,\eta_n),&\quad \theta\in[-\tau,-h), \\[1mm]
&-\theta F(t_n,\eta_n),&\quad \theta\in[-h,0].
\end{array}
\right.\\[1mm]
\eta_{n+1}(\theta)&=&\left\{\setlength\arraycolsep{0.1em}\begin{array}{rll}
&\eta_n(h+\theta),&\quad \theta\in[-\tau,-h), \\[1mm]
&F(t_n,\eta_n),&\quad \theta\in[-h,0].
\end{array}
\right.
\end{array}
\right.
\end{equation*}
Similarly, the exponential Heun method \eqref{heunButcher}
is given by
\begin{equation*}
\left\{\setlength\arraycolsep{0.1em}\begin{array}{rll}
u_{n+1}(\theta)&=&\left\{\setlength\arraycolsep{0.1em}\begin{array}{rll}
&u_n(h+\theta)+\displaystyle\frac{h}{2}(F(t_n,\eta_n)+F(t_{n+1},\eta_{2,n})),&\quad \theta\in[-\tau,-h), \\[1mm]
&\displaystyle \frac{\theta^2}{2h} F(t_n,\eta_n)-\left(\theta+\frac{\theta^2}{2h}\right)F(t_{n+1},\eta_{2,n}),&\quad \theta\in[-h,0],
\end{array}
\right.\\[1mm]
\eta_{n+1}(\theta)&=&\left\{\setlength\arraycolsep{0.1em}\begin{array}{rll}
&\eta_n(h+\theta),&\quad \theta\in[-\tau,-h), \\[1mm]
&\displaystyle -\frac{\theta}{h} F(t_n,\eta_n)+\left(1+\frac{\theta}{h}\right)F(t_{n+1},\eta_{2,n}),&\quad \theta\in[-h,0],
\end{array}
\right.\\[1mm]
\eta_{2,n}(\theta)&=&\left\{\setlength\arraycolsep{0.1em}\begin{array}{rll}
&\eta_n(h+\theta),&\quad \theta\in[-\tau,-h), \\[1mm]
&F(t_n,\eta_n),&\quad \theta\in[-h,0],
\end{array}
\right.
\end{array}
\right.
\end{equation*}
while the method \eqref{order3Butcher} by
\begin{equation*}
\left\{\setlength\arraycolsep{0.1em}\begin{array}{rll}
u_{n+1}(\theta)&=&\left\{\setlength\arraycolsep{0.1em}\begin{array}{rll}
&u_n(h+\theta)+\displaystyle\frac{h}{4}F(t_n,\eta_n)+\frac{3}{4}hF\left(t_n+\frac{2}{3}h,\eta_{3,n}\right),&\quad \theta\in[-\tau,-h), \\[1mm]
&\displaystyle \frac{1}{2}\left(\theta+3\frac{\theta^2}{2h}\right) F(t_n,\eta_n)-\frac{3}{2}\left(\theta+\frac{\theta^2}{2h}\right)F\left(t_n+\frac{2}{3}h,\eta_{3,n}\right),&\quad \theta\in[-h,0],
\end{array}
\right.\\[1mm]
\eta_{n+1}(\theta)&=&\left\{\setlength\arraycolsep{0.1em}\begin{array}{rll}
&\eta_n(h+\theta),&\quad \theta\in[-\tau,-h), \\[1mm]
&-\displaystyle\frac{1}{2}\left(1+3\frac{\theta}{h}\right)F(t_n,\eta_n)+\frac{3}{2}\left(1+\frac{\theta}{h}\right) F\left(t_n+\frac{2}{3}h,\eta_{3,n}\right),&\quad \theta\in[-h,0],
\end{array}
\right.\\[1mm]
\eta_{3,n}(\theta)&=&\left\{\setlength\arraycolsep{0.1em}\begin{array}{rll}
&\eta_n(\frac{2}{3}h+\theta),&\quad \theta\in[-\tau,-\frac{2}{3}h), \\[1mm]
&\displaystyle\left(1+\frac{2\theta}{h}\right)F(t_n,\eta_n)-2\left(1+\frac{\theta}{h}\right))F\left(t_n+\frac{1}{2}h,\eta_{2,n}\right),&\quad \theta\in[-\frac{2}{3}h,0].
\end{array}
\right.\\[1mm]
\eta_{2,n}(\theta)&=&\left\{\setlength\arraycolsep{0.1em}\begin{array}{rll}
&\eta_n(\frac{1}{2}h+\theta),&\quad \theta\in[-\tau,-\frac{1}{2}h), \\[1mm]
&F(t_n,\eta_n),&\quad \theta\in[-\frac{1}{2}h,0].
\end{array}
\right.
\end{array}
\right.
\end{equation*}
Theorem \ref{th_convergence} gives us the order of convergence of the three methods, which also applies to RE. However, one could be interested in the error on $x_t=j^{-1}u$. As for Euler's method \eqref{eulerButcher} by applying $j^{-1}$ to both sides of \eqref{errEul}, we get
\begin{equation*}
\setlength\arraycolsep{0.1em}\begin{array}{rcll}
x_{t_{n+1}}-\eta_{n+1}&=&\tilde{\mathcal{T}_0}(h)(x_{t_n}-\eta_n)+hj^{-1}\varphi_1(h\mathcal{A}_0)(F(t_n,x_{t_n})-F(t_n,\eta_n))H+j^{-1}\delta_n\\[1mm]
&=&\tilde{\mathcal{T}_0}(h)(x_{t_n}-\eta_n)+\left\{\setlength\arraycolsep{0.1em}\begin{array}{rll}&(F(t_n,x_{t_n})-F(t_n,\eta_n))H+(h+\theta)f'(\xi),&\quad \theta\in[-h,0], \\[1mm]
&0,&\quad \theta\in[-\tau,-h].
\end{array}
\right.
\end{array}
\end{equation*}
where $\tilde{\mathcal{T}_0}=j^{-1}\mathcal{T}_0j$. Thus, if we evaluate the $X$-norm of the left-hand side, we get
$$
\|E_{n+1}\|_X\leq\|E_n\|_X+hL\|e_n\|_X+Ch^2
$$
for $E_n:=x_{t_n}-\eta_n$ and some constant $C_1$. Since we know that $\|e_n\|\leq C_2h$ for some constant $C_2$, we get $\|E_{n+1}\|_X=\|E_n\|_X+\mathcal{O}(h^2)$, which gives $\|E_n\|_X=Ch$ for some constant $C$, meaning that the error on $x_t$ decays at the same rate as that on $jx_t$. The argument can be extended to the Heun method \eqref{heunButcher}, as well as other methods satisfying the required order conditions in the strong form.\\

We conclude by remarking that for REs the computation of $F$ involves an integral and so, in general, one has to provide an approximation $\tilde F$. In particular, it is necessary to resort to a quadrature rule, just like with DDEs with distributed delays (see Remark \ref{distrdel}). The resulting ExpRK method are connected to a class of RK method for Volterra equation with delay with continuous extensions \cite{BD93,BJVZ89,B17}.  Even so, ExpRK methods for REs provide a novel, more general approach to derive high-order methods and for convergence analysis.

\section{Numerical results}\label{S_nDDE}
In this section we show the results of numerical simulations using ExpRK-methods on a prototype models defined by DDEs and REs, using constant $h$ such that $\frac{\tau}{h}$ is an integer.  All the relevant MATLAB codes will be freely available at \url{http://cdlab.uniud.it/software} upon publication. First, we consider the IVP \cite[(1.2.3)]{belzen03}
\begin{equation}\label{belzen}
\left\{
\begin{aligned}
&x'(t)=\lambda x(t)-\frac{\pi}{2}\text{e}^{\lambda}x(t-1), \\
&x(t)=\text{e}^{\lambda t}\text{sin}\left(\frac{\pi t}{2}\right),\quad t\in[-1,0].
\end{aligned}
\right.
\end{equation}
Its exact solution is $x(t)=\text{e}^{\lambda t}\text{sin}\left(\frac{\pi t}{2}\right), t\geq -1$. Our goal is to provide numerical evidence of the convergence rate \eqref{errbound} for the methods expEuler \eqref{eulerButcher}, Heun  \eqref{heunButcher} and the method of order 3 \eqref{order3Butcher}.

The solution of \eqref{belzen} is integrated in time from $t=0$ up to $T=2$ using timesteps $h= 10^{-1},10^{-2},\ldots,10^{-6}$. Figure \ref{fig:BZ123_errs} confirms the $O(h^p)$ behavior (being $\mathcal{F}$, as well as the exact solution, infinitely differentiable).

\begin{figure}
%\sidecaption[t]
\centering
\includegraphics[scale=0.75]{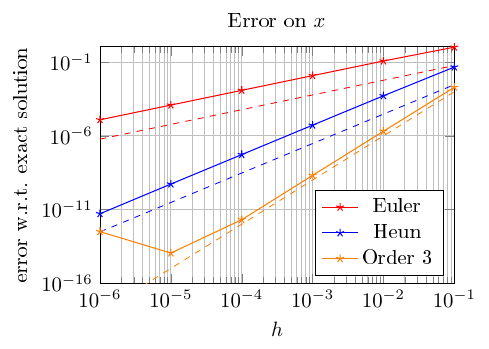}
\caption{Error on the value at $T=2$ of the solution of \eqref{belzen} for $\lambda=1$. Red: Euler method, compared with (dashed) straight line having slope 1. Blue: Heun method, compared with (dashed) straight line having slope 2. Orange: method \eqref{order3Butcher}, compared with (dashed) straight line having slope 3.}
\label{fig:BZ123_errs}
\end{figure}

%\subsection{semidiscretized formulation}\label{SS_sRE}
%Similarly to DDEs, \eqref{eq:RE} can be discretized into a finite-dimensional systems of ODEs via pseudospectral discretization \cite{bdgsv16} on the nodes \eqref{chebII}. In this case, the restriction operator $R_M:X\to\R^M$ maps a function to the vector of its values at the nodes with the exception of $\theta_0=0$ and the prolongation operator $P_{M,0}:\R^M\to X_{M,0}\subset D(\A_0)$ maps the vector to the $M$-degree polynomial interpolating $0$ at the value $0$ and the remaining nodes at the values defining the vector. Assuming again $d=1$ for simplicity, the discretization of \eqref{eq:RE} according to \cite{bdgsv16} reads
%\begin{equation}\label{discrRE}
%    \dot{U}_M=D_M \cdot U_M + F(\cdot,(P_{M,0}(0,U_M)))\mathbf{1}_M ,
%\end{equation}
%where $\mathbf{1}_M\in\R^M$ has all entries equal to $1$.
%Thus, we get
%\begin{equation}\label{discrAFRE}
%A_{0,M}= D_M
%,\qquad F_M(\cdot,U)=\mathcal{F}(\cdot,P_{M,0}(0,U))\mathbf{1}_M.
%\end{equation}

Next, we show the results of numerical simulations using ExpRK-methods on a prototype RE model. Consider the IVP \cite[(A.1)]{bdls16}
\begin{equation}\label{quadraticRE}
\left\{
\begin{aligned}
&x(t)=\displaystyle\frac{\gamma}{2}\int_{t-3}^{t-1}x(s)(1-x(s))\diff s, \\
&x(t)=c+A\text{sin}\left(\frac{\pi t}{2}\right),\quad t\in[-3,0],
\end{aligned}
\right.
\end{equation}
where
\begin{equation*}
\left\{
\begin{aligned}
&c=\displaystyle\frac{1}{2}+\frac{\pi}{4\gamma}, \\
&A^2=\displaystyle\sqrt{2c\left(1-\frac{1}{\gamma}-c\right)}.
\end{aligned}
\right.
\end{equation*}
Its exact solution is $x(t)=c+A\text{sin}\left(\frac{\pi t}{2}\right), t\geq -3$ for certain values of $\gamma$. In the following, we analyze the convergence of our explicit ExpRK-methods \eqref{eulerButcher}, \eqref{heunButcher} and \eqref{order3Butcher} numerically, for \eqref{quadraticRE}.

The solution of \eqref{quadraticRE} is integrated in time from $t=0$ up to $T=4$ using timesteps $h= 10^{-1},10^{-2},\ldots,10^{-5}$. Figure \ref{fig:quadraticRE_errs} confirms the expected $O(h^p)$ behavior for the integrated state $u$, which is maintained on $x$ for the methods satisfying the order conditions in Table \ref{ordercond}. For the method \eqref{order3Butcher}, satisfying the conditions of order $p=3$ only in the weak sense, the order of convergence obtained for $x$ is 2.

\begin{figure}
%\sidecaption[t]
\centering
\includegraphics[scale=0.75]{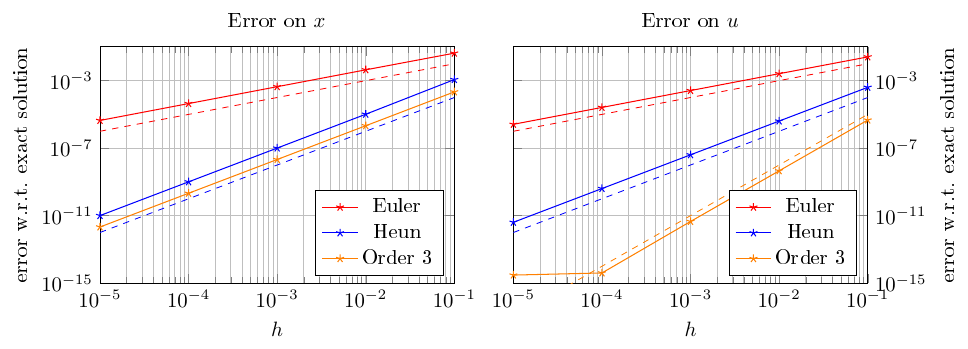}
\caption{Error on the $L^1$-norm at $T=4$ of the solution of \eqref{quadraticRE} for $\gamma=4$ (left) and of its integrated state, defined by \eqref{j_RE} (right). Red: Euler method, compared with (dashed) straight line having slope 1. Blue: Heun method, compared with (dashed) straight line having slope 2. Orange: method \eqref{order3Butcher}, compared in the right plot with (dashed) straight line having slope 3.}
\label{fig:quadraticRE_errs}
\end{figure}
We conclude this section with the results of a numerical simulation showing that the method can also be extended to coupled RE/DDE systems. Namely, we consider the simplified logistic Daphnia model \cite{bdmv13}
\begin{equation}\label{simpledaphnia}
\left\{\setlength\arraycolsep{0.1em}\begin{array}{rcll}
b(t) &=& \beta S(t)\displaystyle\int_{\overline{a}}^{a_{\text{max}}}b(t-a)\diff a,& \\[3mm]
S'(t) &=& r \displaystyle S(t)\left(1-\frac{S(t)}{K}\right)-\gamma S(t)\int_{\overline{a}}^{a_{\text{max}}}b(t-a)\diff a.& 
\end{array}
\right.
\end{equation}
As shown in \cite{bdmv13}, for $r=K=\gamma=1$, $\overline{a}=3$ and $a_{\text{max}}=4$, a Hopf bifurcation occurs when $\beta \approx 3.0162$. The solution of \eqref{simpledaphnia} for $\beta = 3.02$ is integrated in time from $t=0$ up to $T=60$ using timestep $h=10^{-2}$ and the method \eqref{order3Butcher}. The chosen starting initial condition is the constant (0.7, 0.35), which intersect the stable periodic solution. Figure \ref{fig:daphnia} shows both the convergence to the periodic solution and the smoothing effect of the delay, particularly visible on the $b$ component at $t=4$, where the solution has a discontinuous derivative.\\
\begin{figure}
%\sidecaption[t]
\centering
\includegraphics[scale=1.25]{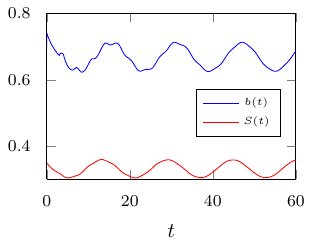}
\caption{Solution of \eqref{simpledaphnia} for $r=K=\gamma=1$, $\overline{a}=3$ and $a_{\text{max}}=4$ from $t=0$ to $T=60$, computed with \eqref{order3Butcher} and $h=10^{-2}$. Red: DDE component. Blue: RE component.}
\label{fig:daphnia}
\end{figure}

\section{Conclusions}\label{S_end}

In this paper, we have seen how ExpRK-methods can be extended from semilinear ODEs to ADEs, whose state space is a space of functions.
The interest towards this class of explicit methods stems from the possibility, given by the sun-star theory, to reformulate delay equations as abstract integral equations. For DDEs, explicit ExpRK-methods reduce to explicit continuous RK methods \cite{BGMZ09,M03}, while for REs they are associated to the class of RK methods for Volterra equations with delay with continuous extensions \cite{BD93,BJVZ89,B17}. However, here we derive them and prove convergence in a unifying and original way. Indeed, the abstract framework includes the coupled RE/DDE systems as well, and the relevant computational results shown in Section \ref{S_nDDE} suggest that we can also extend the convergence theorem accordingly, exploiting the theoretical results available for DDEs and REs separately. The details of the proof will be subject of future work.\\
While we have mostly focused on using constant timesteps, the analysis can be extended to variable timesteps, although this might lead to the loss of one order of convergence for the methods satisfying the conditions of order $p$ only in the weak sense.\\
In the near future, we plan to apply expRK-methods to different classes of equations, such as equations with infinite delay or structured population models, such as \cite{gmc74}.\\
Another interesting extension of the present work would be to explore different exponential methods other than expRK, such as exponential Rosenbrock methods, Magnus methods, linear multistep methods or general linear methods.\\
Finally, we observe that, since ExpRK-methods were originally formulated for ODEs, these methods can alternatively be applied to delay equations after having priorly discretized the latter to finite-dimensional systems of ODEs, e.g., via pseudospectral discretization \cite{bdgsv16}. \\
\section*{Acknowledgments}
The authors are members of INdAM Research group GNCS, as well as of UMI Research group “Modellistica socio-epidemiologica”. This work was supported by the Italian Ministry of University and Research (MUR) through the PRIN 2020 project (No. 2020JLWP23) “Integrated Mathematical Approaches to Socio–Epidemiological Dynamics”, Unit of Udine (CUP G25F22000430006).
%-----------------------------------------------------------------------------
%-----------------------------------------------------------------------------
%-----------------------------------------------------------------------------

%-----------------------------------------------------------------------------
%-----------------------------------------------------------------------------
%-----------------------------------------------------------------------------
\end{document}
%-----------------------------------------------------------------------------
%-----------------------------------------------------------------------------
%-----------------------------------------------------------------------------
\begin{equation}\label{}
A(\tau)= \begin{cases}
A(4-\tau)\geq0, & 0\leq \tau \leq 4, \\
0, & \tau>4,
\end{cases}
\end{equation}

\begin{equation*}
\left\{\setlength\arraycolsep{0.1em}\begin{array}{rcl}
A&B&C\\[2mm]
D&E&F
\end{array}\right.
\end{equation*}

\begin{equation}\label{}
\left\{
\begin{aligned}
A &= (-1)^k \frac{4}{k\pi} \sin \left(\frac{k \pi}{2}\right) \frac{\gamma}{2} (1-2\theta) A, \\
\theta &= \gamma \theta (1-\theta) -\frac{\gamma}{2}A^{2}.
\end{aligned}
\right.
\end{equation}

\begin{align}
\mathcal{K}(\alpha_{1},\alpha_{})&\coloneqq e^{\int_{\alpha_{}}^{\alpha_{1}}g_{1}(\theta)d\theta}g_{}(\alpha_{}),\label{K}\\
\mathcal{Klambda}(\alpha_{1},\alpha_{})&\coloneqq -\mathcal{F}(\alpha_{1},\bar{TB})\left(\int_{\alpha_{}}^{\alpha_{1}}\mu_{1}(\theta)\mathcal{K}(\theta,\alpha_{})d\theta+\mu_{}(\alpha_{})\right).\label{Klambda}
\end{align}

\begin{equation*}
\begin{split}
z(t+4)={}& \int_{0}^4 A(\tau) h(z(t+4-\tau)) {\rm d} \tau \\
={}& \int_{0}^4 A(4-\sigma)h(z(t+\nu)){\rm d} \sigma \\
={}& \int_{0}^4 A(\tau) h(z(t+\tau)) {\rm d} \tau \\
={}& \int_{0}^4 A(\tau) h(z(-t-\tau)) {\rm d} \tau \\
={}& z(-t),
\end{split}
\end{equation*}

\begin{equation}\label{equation1}
\begin{split}
x'&=x(\alpha-\beta y),\\
y'&=-y(\gamma-\delta x).
\end{split}
\end{equation}

\begin{equation}\label{equation2}
\begin{split}
\dot{y}(t) ={}& a y(t)+b f(y(t-\gamma_{0}\tau)))\\
>{}& a y(t)-b f(y(t-\gamma_{0}\tau))\\
& +b f(y(t-\gamma_{0}\tau))\\
={}& a y(t).
\end{split}
\end{equation}

\begin{align}
\dot{y}(t) &= a y(t)+b f(y(t)),\label{equation3a}\\
\ddot{y}(t) &= a y(t) - b f(y(t)),\label{equation3b}\\
y(0) &= y_{0}.\label{equation3c}
\end{align}

\begin{equation}\label{equation4}
|x|=\begin{cases}
x,\quad &x \geq 0,\\
-x, \quad &x<0.
\end{cases}
\end{equation}

\begin{equation*}
\begin{split}
\alpha x(t)+\beta y(t)= {} &\int_{t_{0}}^{t}f(x(s),y(s-\tau))\,ds-\int_{t_{0}}^{t}g(x(s-\tau),y(s))\,ds\\
&+\int_{t_{0}}^{t}h(x(s-\tau),y(s-\tau))\,ds.
\end{split}
\end{equation*}

\begin{figure}[!h]
\centering
\includegraphics[width=\textwidth]{}
\caption{caption. See text for more details.}\label{f_E}
\end{figure}

%-----------------------------------------------------------------------------
%-----------------------------------------------------------------------------
\section{}
\label{s_}
%-----------------------------------------------------------------------------
\subsection{}
\label{s_}
%-----------------------------------------------------------------------------
%-----------------------------------------------------------------------------
\appendix
\section{}
\label{}
%-----------------------------------------------------------------------------
{\color{red} ()}